\documentclass[a4paper,12pt,twoside]{amsart}

\usepackage[cm]{fullpage}

\usepackage[all]{xy}
\usepackage{mymacros,oitp}

\DeclareMathOperator{\Hilb}{Hilb}
\DeclareMathOperator{\im}{Im}

\title[Projective reconstruction in algebraic vision]
{Projective reconstruction in algebraic vision}

\author[A.~Ito]{Atsushi Ito}
\address{
Graduate School of Mathematics,
Nagoya University,
Furocho, Chikusaku, Nagoya,
464-8602,
Japan.
}
\email{atsushi.ito@math.nagoya-u.ac.jp}

\author[M.~Miura]{Makoto Miura}
\address{
Korea Institute for Advanced Study,
85 Hoegiro,
Dongdaemun-gu,
Seoul,
130-722,
Republic of Korea.
}
\email{miura@kias.re.kr}

\author[K.~Ueda]{Kazushi Ueda}
\address{
Graduate School of Mathematical Sciences,
The University of Tokyo,
3-8-1 Komaba,
Meguro-ku,
Tokyo,
153-8914,
Japan.}
\email{kazushi@ms.u-tokyo.ac.jp}
\date{}
\begin{document}

\maketitle
\begin{abstract}
 We discuss the geometry of rational maps from a projective space 
 of an arbitrary dimension to the product of projective spaces 
 of lower dimensions induced by linear projections.
 In particular,
 we give an algebro-geometric variant of the projective
reconstruction theorem by
Hartley and Schaffalitzky
 \cite{Hartley2009}.
\end{abstract}

\section{Introduction}\label{sc:introduction}
Let $r$ be a positive integer and
$\bsm = (m_1, \ldots, m_r)$ be a sequence of positive integers.
For each $i = 1, \ldots, r$,
take a vector space $W_i$ of dimension $m_i+1$
over a field $\bfk$,
which we assume to be an
algebraically closed field of characteristic zero unless otherwise stated%
\footnote{
This condition on the field $\bfk$
allows us to use standard tools in complex algebraic geometry,
such as the theorem of Bertini.
From the viewpoint of application to computer vision,
where the motivation for this paper comes from,
the case $\bfk = \bR$ is of particular interest.
As we mention later in this section,
the reconstruction theorem for $\bfk = \bR$
follows from the reconstruction theorem for $\bfk = \bC$.
}.
Let further $V$ be a vector space
satisfying $n \coloneqq \dim V - 1 > m_i$ for any $i =1, \ldots, r$.
A sequence $\bss = (s_1, \ldots, s_r)$ of surjective linear maps
\begin{align}
 s_i \colon V \to W_i, \quad i = 1, \ldots, r
\end{align}
induce rational maps
\begin{align}
 \varphi_i \colon \bP^n \dto \bP^{m_i}, \quad i = 1, \ldots, r
\end{align}
from $\bP^n \coloneqq \bP(V)$ to $\bP^{m_i} \coloneqq \bP(W_i)$.
We call these rational maps \emph{cameras},
with the model of a pinhole camera as a linear projection in mind.
Correspondingly,
the loci
\begin{align}
 Z_i \coloneqq \bP(\ker s_i) \subset \bP^n, \quad i = 1, \ldots, r
\end{align}
of indeterminacy are called the \emph{focal loci} of the cameras.
The closure $X$ of the image of the rational map
\begin{align}
 \bsvarphi \coloneqq 
 (\varphi_1, \dots, \varphi_r)
  \colon \bP^n \dto \bP^{\bsm} \coloneqq \prod_{i=1}^r \bP^{m_i}
\end{align}
is called the \emph{multiview variety}
in the case $n = 3$ and $\bsm = (2^r)$ 
in \cite{MR3095002},
and we use the same terminology for arbitrary $n$ and $\bsm$.
In this section, 
we assume
$|\bsm| \coloneqq m_1 + \cdots + m_r \geq n+1$,
so that the multiview variety $X$ is a proper subvariety of $ \bP^{\bsm}$.
Basic properties of multiview varieties are studied in \cite{1310.8453},
where formulas for dimensions, multidegrees, and Hilbert polynomials are obtained and the
Cohen--Macaulay property is proved.

The \emph{projective reconstruction problem} asks
if $\bsvarphi$ is determined uniquely from $X$, 
up to the inevitable ambiguity
by the action of $\PGL(n + 1,\bfk)$.
In real-life applications
where $\bfk = \bR$, $n = 3$, and $\bsm = (2^r)$,
this problem may be phrased as follows:
Assume that one is given multiple pictures of one place,
taken with various cameras
whose positions and angles are not known at the beginning.
A \emph{point correspondence} is a collection of points in the pictures,
consisting of one point in every picture,
which is the image of the same point in the place.
The set of point correspondences form an open subset of $X$.
Assume that
one can tell sufficiently many point correspondences,
say, from the features of the objects,
so that one can fix $X$ uniquely.
Now the problem is whether one can `reconstruct'
the 3-dimensional configuration of objects in the place
appearing on the pictures, together with the configuration of the cameras.

Each camera
$
 \varphi_i \colon \bP^n \dto \bP^{m_i}
$
is parametrized by
an open subset of
the projective space
$\bP(V^{\vee} \otimes W_i)$,
where the corresponding linear map $s_i$ has the full rank.
Let
\begin{align}
\label{eq:Phi}
 \Phi \colon
  \prod_{i=1}^r \bP \lb V^{\vee} \otimes W_i \rb
   \dashrightarrow  \Hilb \lb \bP^{\bsm} \rb  
\end{align}
be the rational map
sending a camera configuration $\bsvarphi$
to the multiview variety
$
 X
$
considered as a point
in the Hilbert scheme $\Hilb \lb \bP^{\bsm} \rb$ of subschemes of $\bP^{\bsm}$.
The natural right action of $\Aut (\bP^n) = \PGL(n+1, \bfk)$
on each $\bP \lb V^{\vee} \otimes W_i \rb$
induces the diagonal action on the product
$
 \prod_{i=1}^r \bP \lb V^{\vee} \otimes W_i \rb.
$
The main result of this paper is the following:

\begin{theorem}\label{th:main}
\begin{enumerate}
 \item \label{it:general}
If $\bsm \ne (1^{n+1}) \coloneqq (1, \ldots, 1)$,
then a general fiber of $\Phi$ consists of a single $\PGL(n+1, \bfk)$-orbit.
 \item \label{it:tolines}
If $\bsm = (1^{n+1})$,
then a general fiber of $\Phi$ consists of two $\PGL(n+1, \bfk)$-orbits.
 \item \label{it:dominance}
If $ | \bsm | \geq 2n-1$,
then $\Phi$ is dominant onto an irreducible component of $ \Hilb (\bP^{\bsm} )$.
\end{enumerate}
\end{theorem}

As we see in \pref{sc:hartley}, 
\pref{th:main}.(\pref{it:general}),(\pref{it:tolines})
is an algebro-geometric 
variant of the projective reconstruction theorem by
Hartley and Schaffalitzky \cite{Hartley2009},
with a new purely algebro-geometric proof.
A result closely related to
\pref{th:main}.(\pref{it:dominance})
for $n = 3$ and $\bsm = (2^r)$
is proved in \cite[Theorem 6.3]{MR3095002}.
They use the rational map
\begin{align}
\gamma \colon G(n+1,\bigoplus_{i=1}^r W_i)/\!\!/\bG_m^r 
			\dto \scH_{n,\bsm}
\end{align}
instead of \pref{eq:Phi},
where
$G(n+1,\bigoplus_{i=1}^r W_i)/\!\!/\bG_m^r$ denotes
a GIT quotient of the Grassmannian of $(n+1)$-dimensional subspaces in 
$\bigoplus_iW_i$ by an algebraic torus, and $\scH_{n,\bsm}$ denotes
the multigraded Hilbert scheme parametrizing $\bZ^r$-homogeneous
ideals in the homogeneous coordinate ring of $\bP^\bsm$.
Their result states that the rational map $\gamma$ dominates
an irreducible component of $\scH_{n,\bsm}$.
The bound $ | \bsm | \ge 2 n - 1$ is sharp in this case.
The embedding of the space of cameras
into $\Hilb(\bP^\bsm)$ for $\bsm = (2^r)$ is also discussed in \cite{2017arXiv170709332L}.

Note that a real configuration of cameras can naturally be viewed
as a complex configuration of cameras,
and a pair of real configurations of cameras are related by an action of $\PGL(n+1, \bR)$
if and only if they are related by an action of $\PGL(n+1,\bC)$.
It follows that the reconstruction over $\bC$
in \pref{th:main}.(\pref{it:general})
implies the reconstruction over $\bR$.
See also \pref{rm:overR} for the fact that
\pref{th:main}.(\pref{it:tolines}) also holds over $\bR$.

This paper is organized as follows:
In \pref{sc:projection},
we collect basic facts about linear projections.
The statements (\pref{it:general}),
(\pref{it:tolines}), and
(\pref{it:dominance})
in \pref{th:main}
are proved in Sections \ref{sc:pr1},
\ref{sc:pr2}, and
\ref{sc:dominance} respectively.
In \pref{sc:hartley},
we clarify the relation of \pref{th:main} to results by Hartley and Schaffalitzky.

\begin{acknowledgments}
We thank the anonymous referee for suggestions for improvements.
A.~I.~was supported by Grant-in-Aid 
for Scientific Research (14J01881, 17K14162).
M.~M.~was supported by Korea Institute for Advanced Study.
K.~U.~was partially supported by Grant-in-Aid for Scientific Research
(15KT0105,
16K13743,
16H03930).
\end{acknowledgments}

\section{The geometry of linear projections}
 \label{sc:projection}


Let $V$ and $W$ be vector spaces of dimensions $n+1$ and $m+1$ respectively with $n > m$, and
set $\bP^n \coloneqq \bP(V)$ and $\bP^{m} \coloneqq \bP(W)$.
Let further
$p \colon \bP^n \times \bP^m \rightarrow \bP^n$ and
$q \colon \bP^n \times \bP^m \rightarrow \bP^m$
be the projections to the first and the second factors.
A surjective linear map
$
 s \colon V \rightarrow W
$
induces a \emph{linear projection}
\begin{align}
 \varphi \colon \bP^n \dto \bP^{m},
\end{align}
which is a dominant rational map
from $\bP^n$ to $\bP^m$.

The locus
$
 Z \coloneqq \bP(\ker s) \subset \bP^n
$
of indeterminacy of $\varphi$
can be eliminated by the blow-up
\begin{align}
 \ptilde \colon \Xtilde \coloneqq \Bl_Z \bP^n \to \bP^n
\end{align}
along $Z$,
i.e., there exists a morphism
$
 \qtilde \colon \Xtilde \to \bP^m
$
making the diagram
\begin{align} \label{eq:diagram1}
\begin{aligned}
\xymatrix{
 & \Xtilde \ar[ld]_{\ptilde} \ar[rd]^{\qtilde} & \\ 
 \bP^n \ar@{-->}[rr]^{\varphi} &  & \bP^m  \\
}
\end{aligned}
\end{align}
commutative.

The exceptional divisor
$
 E
  \coloneqq \ptilde^{-1}(Z)
  \subset \Xtilde
$
is the $\bP^m$-bundle over $Z$
obtained as
the projectivization
of the normal bundle
$N_{Z/\bP^n} \cong \cO_{\bP^n}(1)^{\oplus {m+1}} $
of $Z$ in $\bP^n$.
For a hyperplane $H \subset \bP^n$
containing $Z$,
the closure of the image of $H$ by $\varphi$ is a hyperplane $H'$ in $\bP^m$.
Since $\ptilde^* H=\qtilde^* H' +E$,
we have
\begin{align} \label{eq:OE}
 \cO_{\Xtilde}(E) \cong \ptilde^* \cO_{\bP^n} (H) \otimes  \qtilde^* \cO_{\bP^m} (-H') \cong \ptilde^* \cO_{\bP^n} (1) \otimes  \qtilde^* \cO_{\bP^m} (-1).
\end{align}
The morphism 
\begin{align}
 \ptilde \times \qtilde \colon \Xtilde \to \bP^n \times \bP^m
\end{align}
is an embedding,
which allows one to identify $\Xtilde$ with the closure of the graph of the rational map $\varphi$.
Under this embedding,
the morphisms $\ptilde$ and $\qtilde$ are restrictions of the projections $p$ and $q$ respectively.

The Euler sequence
\begin{align} \label{eq:Eulerm}
0 \rightarrow \cO_{\bP^{m}}(-1) \rightarrow W \otimes  \cO_{\bP^{m}} \rightarrow T_{\bP^{m}}(-1) \rightarrow 0
\end{align}
gives
$
 H^0\lb T_{\bP^{m}}(-1)\rb \cong W,
$
which together with 
$
 H^0(\cO_{\bP^n}(1)) \cong V^{\dual}
$
shows
\begin{align}\label{eq:sectionnm}
 H^0 \lb p^* \cO_{\bP^n}(1) \otimes q^* T_{\bP^{m}}(-1) \rb
  \cong H^0 \lb \cO_{\bP^n}(1) \rb \otimes H^0 \lb  T_{\bP^{m}}(-1) \rb
  \cong V^\dual \otimes W.
\end{align}
Let $s_{\bP}$ be the element of
$
 H^0 \lb p^* \cO_{\bP^n}(1) \otimes q^* T_{\bP^{m}}(-1) \rb
$
corresponding to $s \in V^\dual \otimes W$
under the isomorphism \pref{eq:sectionnm}.
Since $T_{\bP^m}(-1)$ is the universal quotient bundle on $\bP^m$,
a point $(\ell_1, \ell_2) \in \bP^n \times \bP^m$
(i.e., a pair of one-dimensional subspaces $\ell_1 \subset V$ and $\ell_2 \subset W$)
is in $s_{\bP}^{-1}(0)$
if and only if $s(\ell_1) \subset \ell_2$.
In other words,
the zero locus of $s_{\bP}$ is precisely the graph $\Xtilde$ of $\varphi$;
\begin{align} \label{eq:sP0}
 s_{\bP}^{-1}(0)
  = \Xtilde
  \subset \bP^n \times \bP^m.
\end{align}
The pull-back of the Euler sequence \pref{eq:Eulerm} to $\Xtilde$
tensored with $\ptilde^* \cO_{\bP^n}(1)$ gives
\begin{align} \label{eq:EulerXt}
 0
  \to \ptilde^* \cO_{\bP^n}(1) \otimes \qtilde^*\cO_{\bP^m}(-1)
  \to \ptilde^* \cO_{\bP^n}(1) \otimes W
  \to \ptilde^* \cO_{\bP^n}(1) \otimes \qtilde^* T_{\bP^m}(-1)
  \to 0.
\end{align}
Let
$
 s_{\Xtilde} \in H^0 \lb \Xtilde, \ptilde^* \cO_{\bP^n}(1) \otimes W \rb
$
be the section corresponding to
$
 s \in V^\dual \otimes W
$
under the isomorphism
\begin{align}
 H^0 \lb  \Xtilde, \ptilde^* \cO_{\bP^n}(1) \otimes W \rb
  &\cong H^0 \lb \bP^n, \ptilde_* \ptilde^* \cO_{\bP^n}(1) \rb \otimes W \\
  &\cong H^0 \lb \bP^n, \cO_{\bP^n}(1) \rb \otimes W \\
  &\cong V^\dual \otimes W.
\end{align}
The section $s_{\Xtilde}$ lies in the image of the injection
\begin{align}
 H^0 \lb \Xtilde, \ptilde^* \cO_{\bP^n}(1) \otimes \qtilde^*\cO_{\bP^m}(-1) \rb
  \hookrightarrow H^0 \lb \Xtilde, \ptilde^* \cO_{\bP^n}(1) \otimes W \rb
\end{align}
induced by \pref{eq:EulerXt},
since its image by the map
\begin{align}
 H^0 \lb  \Xtilde, \ptilde^* \cO_{\bP^n}(1) \otimes W \rb
  \to H^0 \lb  \Xtilde, \ptilde^* \cO_{\bP^n}(1) \otimes \qtilde^* T_{\bP^m}(-1) \rb
\end{align}
induced by \pref{eq:EulerXt}
is zero by \pref{eq:sP0}.
It follows from \pref{eq:OE} and
$
 h^0 \lb \cO_{\Xtilde}(E) \rb = 1
$
that $s_{\Xtilde}^{-1}(0) = E$.

\section{Projective reconstruction
in the case 
$\pmb{m} \ne (1^{n+1})$
}
\label{sc:pr1}

We keep the same notations as in \pref{sc:introduction},
and write the projections as
$
 p \colon \bP^n \times \bP^\bsm \to \bP^n,
$
$
 q_i \colon \bP^n \times \bP^\bsm \to \bP^{m_i},
$
$
 \bsq \coloneqq (q_1, \ldots, q_r) \colon \bP^n \times \bP^\bsm \to \bP^\bsm
$
and
$
 \varpi_i \colon \bP^\bsm \to \bP^{m_i}.
$
We do not assume $|\bsm| \geq n+1$ unless otherwise stated.
The Euler sequence
\begin{align}
0 \rightarrow \cO_{\bP^{m_i}}(-1) \rightarrow W_i \otimes  \cO_{\bP^{m_i}} \rightarrow T_{\bP^{m_i}}(-1) \rightarrow 0
\end{align}
gives $H^0\lb T_{\bP^{m_i}}(-1) \rb=W_i $,
which together with
$
 H^0(\cO_{\bP^n}(1))=V^{\dual}
$
shows
\begin{align}\label{eq:section}
 H^0 \lb \bP^n \times \bP^\bsm, \bigoplus_{i=1}^r p^* \cO_{\bP^n}(1) \otimes q_i^* T_{\bP^{m_i}}(-1) \rb
  = \bigoplus_{i=1}^r V^\dual \otimes W_i \cong \bigoplus_{i=1}^r \Hom (V, W_i).
\end{align}
We abuse notation and identify $\bss \in  \bigoplus_{i=1}^r \Hom (V, W_i)$
with the corresponding global section of $  \bigoplus_{i=1}^r p^* \cO_{\bP^n}(1) \otimes q_i^* T_{\bP^{m_i}}(-1) $ on $\bP^n \times \bP^\bsm$.
Let $\Xtilde \subset \bP^n \times \bP^\bsm$ be 
the closure of the graph of the rational map
$
 \bsvarphi \colon \bP^n \dto \bP^\bsm,
$
and set
$
 \ptilde \coloneqq p|_\Xtilde,
$
$
 \qtilde_i \coloneqq q_i|_\Xtilde,
$
and
$
 \bsqtilde \coloneqq \bsq|_\Xtilde,
$
so that we have the diagram
\begin{align}
				\label{eq:xtilde}
\begin{aligned}
\xymatrix{
 & \Xtilde \ar[ld]_{\ptilde} \ar[rd]^{\bsqtilde} & \\ 
 \bP^n \ar@{-->}[rr]^{\bsvarphi} & & X \ar@{}[r]|*{\subset} & \bP^\bsm.  \\
}
\end{aligned}
\end{align}

In all the statements,
such as lemmas, propositions, and theorems,
through the rest of this paper,
we assume that the section $\bss$
is general.

\begin{lemma}\label{lm:fiber}
Let $Z_\bss \subset  \bP^n \times \bP^\bsm$ be the zero locus
of the section $\bss $ in \pref{eq:section}.
For $z \in \bP^n$,
{\black one has}
\begin{align}\label{eq:fiber}
Z_\bss \cap p^{-1}(z)
 =  \{ z\} \times \prod_{ i \colon z \nin Z_i} \varphi_i (z) \times \prod_{ i \colon z \in Z_i}  \bP^{m_i} .
\end{align}
\end{lemma}

\begin{proof}
Let $v \in V$ be a vector corresponding to $z \in \bP^n=\bP(V)$.
Then 
$Z_\bss \cap p^{-1}(z) \subset \{z\} \times  \bP^\bsm$ coincides with the zero locus of the section
\begin{align}
(s_1(v), \dots,s_r(v)) \in \bigoplus_{i=1}^r  W_i  \cong H^0 \lb  \bP^\bsm, \bigoplus_{i=1}^r \varpi_i^* T_{\bP^{m_i}}(-1) \rb.
\end{align}
For each $i$,
the zero locus of $s_i(v) \in H^0(\bP^{m_i}, T_{\bP^{m_i}}(-1))$ is $\varphi_{i}(z)$ (resp.\ $\bP^{m_i}$) if $s_i(v) \neq 0$ (resp.\ $s_i(v) =0$).
Since $s_i(v) =0$ if and only if $z \in Z_i$, we have \pref{eq:fiber}.
\end{proof}

\begin{lemma} \label{lm:Zs=X}
The zero
locus $Z_\bss$
coincides with $\Xtilde$.
\end{lemma}

\begin{proof}
Since $\bP^n$ is irreducible,
the closure $\Xtilde$ of the graph of $\bsvarphi$ is irreducible of dimension $n$.
By \pref{lm:fiber},
the generic point of $\Xtilde$ and hence $\Xtilde$ itself
are contained in $Z_\bss$.
Euler sequences show that
$T_{\bP^{m_i}}(-1) $ are globally generated,
and hence so is $\bigoplus_{i=1}^r p^* \cO_{\bP^n}(1) \otimes q_i^* T_{\bP^{m_i}}(-1)$.
Since $\bss$ is a general section
of a globally generated bundle,
the zero of $\bss$ is smooth of dimension $n$
by a generalization of the theorem of Bertini
\cite[Theorem 1.10]{MR1201388}.
Since $Z_\bss$ is smooth and all fibers of $Z_\bss \rightarrow \bP^n$ are irreducible by \pref{lm:fiber},
$Z_\bss$ is irreducible as well.
 It follows that $\Xtilde$ and $Z_\bss$ are equal,
 since $\Xtilde \subset Z_\bss$ and they are irreducible of the same dimension.
\end{proof}

We regard the section
$
 \bsq_* \bss \in H^0 \lb \bP^\bsm, V^\dual \otimes \bigoplus_{i=1}^r \varpi_i^* T_{\bP^{m_i}}(-1) \rb
$
as a morphism
\begin{align}
 V \otimes \cO_{\bP^\bsm} \to \bigoplus_{i=1}^r \varpi_i^* T_{\bP^{m_i}}(-1)
\end{align}
on $ \bP^\bsm$.
It follows from the definition of $Z_{\bss}=\Xtilde$ that
\begin{align} \label{eq:fiber_q}
 \bsqtilde^{-1}(x) = 
\bP\lb \ker \lb \bsq_* \bss \otimes k(x) \rb \rb
 \times \{x\} \subset \bP^n \times \bP^\bsm
\end{align}
for any $x \in \bP^\bsm$,
where 
 \[
 \bsq_* \bss \otimes k(x)
 \colon V  \to \bigoplus_{i=1}^r \varpi_i^* T_{\bP^{m_i}}(-1) 
 \otimes 
 k(x) 
 \]
 is the induced linear map tensoring by $k(x)$.
For $0 \le j \le \min \{n+1, |\bsm|\}$,
let
$
 X_j \subset \bP^\bsm 
$
be the \emph{$j$-th degeneracy locus} of
$
 \bsq_* \bss
$
defined as the zero of
\begin{align} \label{eq:degeneracy}
 \lb \bsq_* \bss \rb^{\wedge (j+1)}
  \colon \bigwedge^{j+1} V \otimes \cO_{\bP^\bsm} \to \bigwedge^{j+1} \bigoplus_{i=1}^r \varpi_i^* T_{\bP^{m_i}}(-1),
\end{align}
that is,
the locus where the rank of $\bsq_* \bss  $ is at most $j$.
Hence
one has
\begin{align} \label{eq:q_*s}
 \dim \ker \lb \bsq_* \bss \otimes k(x) \rb
  = n + 1 - j
\end{align}
for $x \in X_j \setminus X_{j-1}$.
In particular,
one has
\begin{align} \label{eq:X=X_n}
 X =X_n.
\end{align}
 As $T_{\bP^{m_i}}(-1)$ is globally generated,
 so is $V^\dual \otimes \bigoplus_{i=1}^r \varpi_i^* T_{\bP^{m_i}}(-1)$.
Since
$
 \bsq_* \bss 
$
is a general section of a globally generated vector bundle,
one has
$X_j=\emptyset$ or
\begin{align}
 \codim \lb X_j \middle/ \, \bP^\bsm \rb = (n+1-j)( | \bsm |- j)
\end{align}
by
\cite[Theorem 2.8]{Ottaviani}.
If $ | \bsm | \ge n+1$,
the dimension of $X_{n-1} $ is at most
\begin{align}\label{eq:dimX_n-1}
| \bsm |- (n+1-(n-1))( | \bsm |-(n-1))
  &=| \bsm |-2( | \bsm |-n+1) \\
  &= 2n- | \bsm |-2 \label{eq:dimX_n-1-2} \\
  &\le n-3.
\end{align}
Since $ \bsqtilde : \Xtilde \to X$ is an isomorphism over $X\setminus X_{n-1}$ by \pref{eq:fiber_q} and $\Xtilde$ is smooth,
$X \setminus X_{n-1}$ is smooth.
Thus $X = X_n$ is smooth in codimension one.
Since $X_n$ is Cohen-Macaulay by \cite[Chapter II]{MR770932},
$X$ is normal. 

A morphism is said to be \emph{small}
if it is an isomorphism in codimension one.

\begin{lemma}\label{lm:small}
If $ | \bsm | \ge n+1$, the morphism $\bsqtilde \colon \Xtilde \to X$ is small.
\end{lemma}

\begin{proof}
It follows from \pref{eq:fiber_q}
that $\bsqtilde$ is an isomorphism over $X_n \setminus X_{n-1}$ and
\begin{align}
 \dim \bsqtilde^{-1} (X_{n-1})
  &= \max \lc \dim \lb X_j \setminus X_{j-1} \rb + n-j \relmid 0 \le j \le n-1 \rc \\
  &\le \max \lc| \bsm |- (n+1-j)( | \bsm |-j) + n-j \relmid 0 \le j \le n-1 \rc \\
  &= \max \lc j-(n-j)( | \bsm |-j-1) \relmid 0 \le j \le n-1 \rc \\
  &= 2n- | \bsm |-1 \\
  &\le n-2,
\end{align}
hence $\bsqtilde$ is small.
\end{proof}

Since
$
 s_i \colon p^* \cO_{\bP^n}(-1) \to q_i^* T_{\bP^{m_i}}(-1)
$
is zero on $\Xtilde$,
we see by a similar argument in \pref{sc:projection} 
that its restriction $s_i |_\Xtilde$ is a global section of
$
 \ptilde^* \cO_{\bP^n}(1) \otimes \qtilde_i^*\cO_{\bP^{m_i}}(-1)
$
by the exact sequence on $\Xtilde$
\begin{align}
0 \rightarrow  \ptilde^* \cO_{\bP^n}(1) \otimes \qtilde_i^*\cO_{\bP^{m_i}}(-1) \rightarrow  \ptilde^* \cO_{\bP^n}(1) \otimes W_i  \rightarrow  \ptilde^* \cO_{\bP^n}(1) \otimes \qtilde_i^* T_{\bP^{m_i}}(-1) \rightarrow 0.
\end{align}
Let $E_i \subset \Xtilde$ be the Cartier divisor defined by this section.

By \pref{eq:sP0}, $\Bl_{Z_i} \bP^n  \subset \bP^n \times \bP^{m_i}$
is the zero locus of $(s_i)_{\bP}$,
where we use the notation in  \pref{sc:projection}.
Since $\Xtilde$ is the zero locus of $\bss=(s_1,\dots,s_r)$,
$\ptilde : \Xtilde \rightarrow \bP^n$ factors as $\Xtilde \rightarrow\Bl_{Z_i} \bP^n  \rightarrow  \bP^n$ by projections $ \bP^n \times \bP^\bsm \to \bP^n \times \bP^{m_i} \to \bP^n$.
As in the last paragraph in \pref{sc:projection},
$s_i$ induces a global section $(s_i)_{\Bl_{Z_i} \bP^n}$ on $\Bl_{Z_i} \bP^n$ whose zero locus in the exceptional divisor,
and $s_i |_\Xtilde$ is nothing but the pullback of $(s_i)_{\Bl_{Z_i} \bP^n}$.
Hence
$E_i$ is the total transform of the exceptional divisor of $\Bl_{Z_i} \bP^n$.
In particular, $E_i = \ptilde^{-1}(Z_i) $ holds.

\begin{lemma}\label{lm:exc_div}
\begin{enumerate}
\item The restriction of $\ptilde$ to
$
 \Xtilde \setminus \bigcup_{i=1}^r E_i
$
is an isomorphism onto
$
 \bP^n \setminus \bigcup_{i=1}^r Z_i.
$
\item For each $i = 1, \ldots, r$,
the divisor $E_i$ is irreducible.
\item One has
$
 \bsqtilde(E_1) = \bP^{m_1} \times \overline{(\varphi_2,\dots,\varphi_r)(Z_1)},
$
and similarly for $\bsqtilde(E_i)$ for $i = 2, \ldots, r$.
\end{enumerate}
\end{lemma}

\begin{proof}
(1) holds since the rational map $ \bsvarphi$ is defined on
 $
 \bP^n \setminus \bigcup_{i=1}^r Z_i
$
and $\Xtilde$ is the graph of $ \bsvarphi$.
\\
(2)
 It 
follows from \pref{lm:fiber} that
$\ptilde^{-1}(Z_i \setminus \bigcup_{j \neq i} Z_j)$ is irreducible of dimension $n-1$
and
$
 \dim \ptilde^{-1}(Z_i \cap \bigcup_{j \neq i} Z_j) < n-1.
$
Since $E_i$ is a Cartier divisor,
all irreducible components are $n-1$-dimensional,
and hence $E_i$ is irreducible.
\\
(3) By \pref{lm:fiber}
  and \pref{lm:Zs=X},
  $\ptilde^{-1}(z) = \Xtilde \cap p^{-1}(z)$ coincides with
  \[
   \lc z \rc \times \bP^{m_1} \times \varphi_2(z)\times  \dots \times \varphi_r(z) 
  \]
  for all $z\in Z_1 \setminus\bigcup_{j \neq 1} Z_j$. Hence
the image of $\ptilde^{-1}(
Z_1
\setminus
\bigcup_{j \neq 1} Z_j
)$ by $ \bsqtilde$ is $ \bP^{m_1} \times (\varphi_2,\dots,\varphi_r)(
Z_1
\setminus
\bigcup_{j \neq 1} Z_j
) $.
Since $\ptilde^{-1}(
Z_1
\setminus
\bigcup_{j \neq 1} Z_j
)$ is dense in $E_1$,
we have (3).
\end{proof}



From now on, we assume $|\bsm| \geq n+1$.
\pref{lm:small}
implies that
$\bsqtilde(E_i)$ is a prime Weil divisor on $X$.
We set
\begin{align} \label{eq:L_i}
 L_i \coloneqq \varpi_i^* \cO_{\bP^{m_i}}(1) |_X
\end{align}
for $i = 1, \ldots, r$.

\begin{lemma}\label{lm:inverse}
The rational map defined by $\left| \cO_X(\bsqtilde(E_1)) \otimes  L_1 \right|$ is inverse
to the rational map $\bsvarphi \colon \bP^n \dto X$ up to $\Aut (\bP^n)$.
\end{lemma}

\begin{proof}
For a linear system $\Lambda$,
we let $\phi_{\Lambda}$ denote
the rational map defined by $\Lambda$.

Since $X$ is normal and $\bsqtilde$ is small,
we have 
\begin{align} \label{eq:1}
 H^0 \lb X,  \cO_X \lb \bsqtilde(E_1) \rb \otimes  L_1 \rb
  = H^0 \lb \Xtilde, \cO_{\Xtilde}(E_1) \otimes \qtilde_1^* \cO_{\bP^{m_1}}(1) \rb
\end{align}
and the rational map $\phi_{\left| \cO_X(\bsqtilde(E_1)) \otimes  L_1 \right|}$
coincides with the composite map $\phi_{\left| \cO_{\Xtilde}(E_1) \otimes \qtilde_1^* \cO_{\bP^{m_1}}(1)  \right| } \circ \bsqtilde^{-1}$.

Since $\cO_{\Xtilde}(E_1) \cong \ptilde^* \cO_{\bP^n}(1) 
\otimes \qtilde_1^*\cO_{\bP^{m_1}}(-1)$,
we have 
$\cO_{\Xtilde}(E_1) \otimes \qtilde_1^*\cO_{\bP^{m_1}}(1) \cong \ptilde^* \cO_{\bP^n}(1) $.
Thus the rational map $\phi_{\left| \cO_{\Xtilde}(E_1) \otimes \qtilde_1^* \cO_{\bP^{m_1}}(1)  \right| }$
coincides with $\phi_{|\ptilde^* \cO_{\bP^n}(1)| } =\ptilde$ up to $\Aut (\bP^n)$.
Hence  $\phi_{\left| \cO_X(\bsqtilde(E_1)) \otimes  L_1 \right|} =\phi_{\left| \cO_{\Xtilde}(E_1) \otimes \qtilde_1^* \cO_{\bP^{m_1}}(1)  \right| } \circ \bsqtilde^{-1}$ 
coincides with $\ptilde \circ \bsqtilde^{-1} = \bsvarphi^{-1}$ up to $\Aut (\bP^n)$.
\end{proof}

\pref{lm:inverse} shows that we can reconstruct $\bsvarphi$
up to $\Aut (\bP^n)$
from $X$ and $\bsqtilde(E_1)$.
\pref{lm:unique} below shows that $\bsqtilde(E_1)$ is uniquely determined by
$
 X \subset \bP^\bsm
$
if $ | \bsm | \ge n+2$ or $ | \bsm | =n+1$ and $m_1 \ge 2$.

\begin{lemma}\label{lm:unique}
Assume $ | \bsm | \ge n+2$ or $ | \bsm | =n+1$ and $m_1 \ge 2$.
Then $\bsqtilde(E_1)$ is the unique Weil divisor on $X$
of the form $\bP^{m_1} \times Y$
for some $Y \subset  \prod_{i \ne 1} \bP^{m_i}$.
\end{lemma}

\begin{proof}
By \pref{lm:exc_div}.(3), $\bsqtilde(E_1)$ is a Weil divisor of such form.

Assume there exists a subvariety $Y  \subset  \prod_{i \ne 1} \bP^{m_i}$ of dimension $n-1-m_1$
such that 
$D=\bP^{m_1} \times Y $ is contained in $ X$ and $D \ne \bsqtilde(E_1)$.
Let
$
 \Xtilde^{\dagger}  \subset \bP^n \times \prod_{i \ne 1} \bP^{m_i}
$
and $X^{\dagger} \subset \prod_{i \ne 1} \bP^{m_i}$
be the subvarieties obtained from $(s_2,\dots,s_r)$
in the same way as $\Xtilde$ and $X$.
Consider the diagram 
\begin{align}
\begin{aligned}
\xymatrix{
   \Xtilde \ar[d]_{\tilde{\pi}} \ar[r]^{\bsqtilde} &  X\ar[d]_{\pi}   \\ 
   \Xtilde^{\dagger}
	 \ar[r]^{\bsqtilde^{\dagger}} &  X^{\dagger}  ,\\
} 
\end{aligned}
\end{align}
 where $\pi,\tilde{\pi}, \bsqtilde^{\dagger}$ are induced by the natural projections.
\pref{lm:exc_div} shows that
$E_1$ is the unique divisor contracted by $\tilde{\pi}$.
Since $\bsqtilde$ is small and $D \ne \bsqtilde(E_1)$,
we have 
a divisor $\Dtilde^{\dagger} \subset \Xtilde^{\dagger}$ , which is the strict transform of $D$
by the birational map $ \tilde{\pi} \circ \bsqtilde^{-1}$.
Since $D=\bP^{m_1} \times Y$,
one has
$\pi(D)=Y$
and hence
$
 \bsqtilde^{\dagger} \lb \Dtilde^{\dagger} \rb = Y.
$

Since $\dim \Dtilde^{\dagger} -\dim Y =m_1$,
it follows from \pref{eq:fiber_q} and \pref{eq:q_*s} for $ \bsqtilde^{\dagger}, \Xtilde^{\dagger} ,X^{\dagger}$
that $Y$ must be contained in $X^{\dagger}_{n-m_1}$,
where we define the degeneracy locus $X^{\dagger}_{n-m_1} $ in the same way as $X_j$.
On the other hand,
the dimension of $X^{\dagger}_{n-m_1}$ is at most
\begin{align}
 \sum_{i \ne 1} m_i - \left(n+1-(n-m_1) \right) \left(\sum_{i \ne 1} m_i - (n-m_1) \right) =\sum_{i \ne 1} m_i  - (m_1+1)( | \bsm |-n). 
\end{align}
Hence we have
\begin{align}
 n-1-m_1 =\dim Y \le \dim X^{\dagger}_{n-m_1} \le \sum_{i \ne 1} m_i  - (m_1+1)( | \bsm |-n),
\end{align}
which implies $m_1( | \bsm |-n) \le 1$.
This contradicts the assumption $ | \bsm | \ge n+2$ or $| \bsm |=n+1$ and $m_1 \ge 2$.
\end{proof}

\begin{proof}[Proof of \pref{th:main}.(\pref{it:general})]
Let $X$ be the multiview variety {\black for} general $\bsvarphi$.
To show that $X \subset \bP^\bsm$
determines the rational map $\bsvarphi$
uniquely up to the action of
$
 \Aut(\bP^n) \cong \PGL(n+1, \bfk),
$
{\black it suffices} to see that the inverse $\bsvarphi^{-1}$ is uniquely determined
by $X \subset \bP^\bsm$ up to $ \PGL(n+1, \bfk)$.

Assume $|\bsm | \geq n+1$ and $\bsm \ne (1^{n+1}) \coloneqq (1, \ldots, 1)$.
Relabeling the indexes of $m_i$ if necessary,
we may assume that $ | \bsm | \ge n+2$ or $ | \bsm | =n+1$ and $m_1 \ge 2$.
Then \pref{lm:unique} states that
$X \subset \bP^\bsm$ uniquely determines
$\bsqtilde(E_1) \subset X$
without using $\bsvarphi$, $\bsqtilde$, etc.
Hence $X \subset \bP^\bsm$ uniquely determines
${\bsvarphi}^{-1}$
by \pref{lm:inverse}.
The inevitable ambiguity by the action of $ \PGL(n+1, \bfk)$ comes from
the identification of $\phi_{\left| \cO_{\Xtilde}(E_1) \otimes \qtilde_1^* \cO_{\bP^{m_1}}(1)  \right| }$ with $\phi_{|\ptilde^* \cO_{\bP^n}(1)| } =\ptilde$.
\end{proof}

\section{Projective reconstruction in the case {$\pmb{m}=(1^{n+1})$}}
 \label{sc:pr2}

Assume $r=n+1$ and $m_i=1$ for any $1 \le i \le n+1$.
Note that $\bP(W_i)$ can be canonically identified
with $\bP(W_i^\dual)$
since $\dim W_i =2$.
Set
\begin{align}
 V^{\prime} \coloneqq (\coker \bss)^\dual,
\end{align}
which is $(n+1)$-dimensional since
$
 \bss = (s_1,\dots,s_{n+1}) \colon V \to \bigoplus_{i=1}^{n+1} W_i
$
is general.
The canonical inclusion
\begin{align}\label{eq_def_s'}
 \bss^{\prime} \colon V^{\prime} \to \lb \bigoplus_{i=1}^{n+1} W_i \rb^\dual
  = \bigoplus_{i=1}^{n+1} W_i^\dual
\end{align}
defines a hypersurface 
\begin{align}
 X^{\prime} \subset \prod_{i=1}^{n+1} \bP(W_i^\dual) =(\bP^1)^{n+1}
\end{align}
in the same way as $X$.
We also define $\Xtilde^{\prime}, E^{\prime}_i,  \bsqtilde^{\prime}$, etc.\ in the same way as $X$.

\begin{lemma} \label{lm:X=X'}
The hypersurfaces $X$ and $X^{\prime}$ coincide
under the canonical identifications $\bP(W_i) = \bP(W_i^\dual)$
for $i = 1, \ldots, n+1$.
\end{lemma}

\begin{proof}
On $\prod_{i=1}^{n+1} \bP(W_i)$,
we have a diagram
\begin{equation}\label{eq:diag_s_s'}
\begin{aligned}
  \xymatrix{    &  & 0 \ar[d]
    \\
    &  & \bigoplus_{i=1}^{n+1} \varpi_i^{*}\cO_{\bP(W_i)}(-1) \ar[rd] \ar[d]
    \\
    0 \ar[r] & V \otimes \cO \ar[r]^(.4){\bss}  \ar[rd]
    & (\bigoplus_{i=1}^{n+1} W_i) \otimes \cO  \ar[r]^(.5){(\bss^{\prime})^\dual} \ar[d] & (V^{\prime})^{\vee } \otimes \cO  \ar[r]  & 0.
    \\
     &    & \bigoplus_{i=1}^{n+1} \varpi_i^* T_{\bP(W_i)}(-1)\ar[d] & & 
    \\
    & & 0
  }
\end{aligned}
\end{equation}
A point
$
 x \in  \prod_{i=1}^{n+1} \bP(W_i)
$
is contained in $X=X_n$ if and only if the rank of
the linear map 
\begin{align}
 \bsq_* \bss \otimes k(x) \colon V  \to  \bigoplus_{i=1}^{n+1} \varpi_i^* T_{\bP(W_i)}(-1) \otimes k(x)
\end{align}
is at most $n$,
that is,
$\bsq_* \bss \otimes k(x)$ is not injective.
By \pref{eq:diag_s_s'},
this is equivalent to the condition
that 
\begin{align}
 \bss(V) \cap  \bigoplus_{i=1}^{n+1} \varpi_i^{*}\cO_{\bP(W_i)}(-1) \otimes k(x) \ne \{0\},
\end{align}
where we take the intersection as subspaces of $\bigoplus_{i=1}^{n+1} W_i$.
By \pref{eq:diag_s_s'} again,
this is equivalent to the condition that
the rank of the linear map
\begin{align}\label{eq:X'}
 \bigoplus_{i=1}^{n+1} \varpi_i^{*}\cO_{\bP(W_i)}(-1) \otimes k(x) \to (V^{\prime})^\dual
\end{align}
is at most $n$.
Under the identification $\bP(W_i^\dual)=\bP(W_i)$,
the sheaf $T_{\bP(W_i^\dual)}(-1)$  is identified with $ \cO_{\bP(W_i)}(1) $.
Hence the rank of the linear map \pref{eq:X'} is at most $n$ if and only if $x$ is contained in $X^{\prime}$.
Thus $X=X^{\prime}$ holds.
\end{proof}

Recall from \pref{lm:exc_div}
that
$
 \bsqtilde(E_1) = \bP(W_1) \times \overline{(\varphi_2,\dots,\varphi_r)(Z_1)}.
$
Similarly, one has
$
 \bsqtilde^{\prime}(E^{\prime}_1)
  = \bP(W_1^\dual) \times \overline{(\varphi^{\prime}_2,\dots,\varphi^{\prime}_r)(Z^{\prime}_1)}.
$

\begin{lemma}\label{lm:q'(E'_1)}
The closure $\overline{(\varphi_2,\dots,\varphi_r)(Z_1)} \subset \prod_{i=2}^{n+1} \bP(W_i)$
is the $(n-2)$-th degeneracy locus of the composite map 
\begin{align}\label{eq:in_lem_q'(E'_1)}
 (\ker s_1) \otimes \cO_{ \prod_{i=2}^{n+1} \bP(W_i) }
  \hookrightarrow  V \otimes \cO_{ \prod_{i=2}^{n+1} \bP(W_i) } 
  \xto{(s_2,\dots,s_{n+1}) } \bigoplus_{i=2}^{n+1} \varpi_i^* T_{\bP(W_i)}(-1),
\end{align}
where 
we use the same letter 
$\varpi_i$
for
the projection $ \prod_{i=2}^{n+1} \bP(W_i) \rightarrow \bP(W_i)$.
On the other hand,
the closure
$
 \overline{(\varphi^{\prime}_2,\dots,\varphi^{\prime}_r)(Z^{\prime}_1)}
  \subset \prod_{i=2}^{n+1} \bP(W_i^\dual) = \prod_{i=2}^{n+1} \bP(W_i)
$
is the $(n-1)$-th degeneracy locus of
\begin{align}\label{eq:2_in_lem_q'(E'_1)}
(s_2,\dots,s_{n+1})  \colon  V \otimes \cO_{ \prod_{i=2}^{n+1} \bP(W_i) }  \to \bigoplus_{i=2}^{n+1} \varpi_i^* T_{\bP(W_i)}(-1).
\end{align}
\end{lemma}

\begin{proof}
The first statement follows by applying \pref{eq:X=X_n} to
$
 (\varphi_2,\dots,\varphi_r) |_{Z_1} \colon Z_1=\bP(\ker s_1) \dto  \prod_{i=2}^{n+1} \bP(W_i).
$
Since $s_1 \colon V \to W_1$ is surjective,
we have an exact sequence 
\begin{align}
 0 \to \ker s_1 \to \bigoplus_{i=2}^{n+1} W_i \to  (V^{\prime})^{\vee } \to 0,
\end{align}
which gives a diagram
\begin{equation}\label{diag_s_s'_2}
\begin{aligned}
  \xymatrix{    &  & 0 \ar[d]
    \\
    &  & \bigoplus_{i=2}^{n+1} \varpi_i^{*}\cO_{\bP(W_i)}(-1) \ar[rd] \ar[d]
    \\
    0 \ar[r] &( \ker s_1 ) \otimes \cO \ar[r]  \ar[rd]
    & (\bigoplus_{i=2}^{n+1} W_i) \otimes \cO  \ar[r] \ar[d] & (V^{\prime})^{\vee } \otimes \cO  \ar[r]  & 0
    \\
     &    & \bigoplus_{i=2}^{n+1} \varpi_i^* T_{\bP(W_i)}(-1)\ar[d] & & 
    \\
    & & 0
  }
\end{aligned}
\end{equation}
on $\prod_{i=2}^{n+1} \bP(W_i)$.
By an argument similar to that in the proof of \pref{lm:X=X'},
we see that  the $(n-2)$-th degeneracy locus of \pref{eq:in_lem_q'(E'_1)} coincides with
the $(n-1)$-th degeneracy locus of
$
 \bigoplus_{i=2}^{n+1} \varpi_i^{*}\cO_{\bP(W_i)}(-1)  \to   (V^{\prime})^{\vee } \otimes \cO,
$
that is,
the $(n-1)$-th degeneracy locus of
$
 V^{\prime}\otimes \cO \to \bigoplus_{i=2}^{n+1} \varpi_i^{*}\cO_{\bP(W_i)}(1)
$
on $\prod_{i=2}^{n+1} \bP(W_i)$.

By replacing $X$ with $X^{\prime}$, 
we see that $ \overline{(\varphi^{\prime}_2,\dots,\varphi^{\prime}_r)(Z^{\prime}_1)} \subset  \prod_{i=2}^{n+1} \bP(W_i)$ is the $(n-1)$-th degeneracy locus of \pref{eq:2_in_lem_q'(E'_1)}
since $V^{\prime}$ and $\varpi_i^{*}\cO_{\bP(W_i)}(1) $ are replaced by $V$ and $\varpi_i^* T_{\bP(W_i)}(-1)$ respectively.
\end{proof}

We note that \pref{eq:2_in_lem_q'(E'_1)} does not depend on $s_i$,
hence neither does $\overline{(\varphi^{\prime}_2,\dots,\varphi^{\prime}_r)(Z^{\prime}_1)}$.

\begin{lemma} \label{lm:E_1_not_E_1'}
One has $ \bsqtilde(E_1) \ne  \bsqtilde^{\prime}(E^{\prime}_1)$. 
\end{lemma}

\begin{proof}
It suffices to see $ \overline{(\varphi_2,\dots,\varphi_r)(Z_1)}  \ne  \overline{(\varphi^{\prime}_2,\dots,\varphi^{\prime}_r)(Z^{\prime}_1)}$.
Take a general point $ y \in   \overline{(\varphi^{\prime}_2,\dots,\varphi^{\prime}_r)(Z^{\prime}_1)}$.
By \pref{lm:q'(E'_1)},
the rank of
\begin{align}
 (s_2,\dots,s_{n+1})_y  \colon  V  \to \bigoplus_{i=2}^{n+1} \varpi_i^* T_{\bP(W_i)}(-1) \otimes k(y) 
\end{align}
is $n-1$ since $s_2, \ldots, s_{n+1}$ and $y$ are general.
Hence $\ker (s_2,\dots,s_{n+1})_y \subset V$ is two-dimensional.
Then
$
 \ker (s_2,\dots,s_{n+1})_y \cap \ker s_1 =\{0\}  \subset V
$
since $\ker s_1\subset V$ is of codimension two and general.
This means that \pref{eq:in_lem_q'(E'_1)} has rank $n-1$ at $y$.
By \pref{lm:q'(E'_1)}, we have $y \not \in \overline{(\varphi_2,\dots,\varphi_r)(Z_1)}$.
\end{proof}

\pref{lm:E_1_not_E_1'} and \pref{lm:2_exc_divs} below show that
we have exactly two reconstructions:

\begin{lemma} \label{lm:2_exc_divs}
The exceptional locus of the birational morphism $X \to \prod_{i=2}^{n+1} \bP(W_i)$
is the union of $ \bsqtilde(E_1)$ and $ \bsqtilde^{\prime}(E^{\prime}_1)$.
\end{lemma}

\begin{proof}

Since $X \subset \bP(W_1) \times \prod_{i=2}^{n+1} \bP(W_i)$,
the exceptional locus of $X \to \prod_{i=2}^{n+1} \bP(W_i)$ is 
\begin{align}
 \bP(W_1) \times \lc y \in  \prod_{i=2}^{n+1} \bP(W_i) \relmid \bP(W_1) \times \{y\} \subset X \rc
  \subset X. 
\end{align}
Hence we need to show 
\begin{align} \label{eq:2_exc_locus}
 \lc y \in  \prod_{i=2}^{n+1} \bP(W_i) \relmid \bP(W_1) \times \{y\} \subset X \rc
  = \overline{(\varphi_2,\dots,\varphi_r)(Z_1)} \cup  \overline{(\varphi^{\prime}_2,\dots,\varphi^{\prime}_r)(Z^{\prime}_1)}.
\end{align}
Since
$
 \bsqtilde(E_1) = \bP(W_1) \times \overline{(\varphi_2,\dots,\varphi_r)(Z_1)}
$
and
$
 \bsqtilde^{\prime}(E^{\prime}_1) = \bP(W_1^\dual) \times \overline{(\varphi^{\prime}_2,\dots,\varphi^{\prime}_r)(Z^{\prime}_1)},
$
the inclusion $\supset$ in \pref{eq:2_exc_locus} is clear.
To show the converse inclusion,
we take
$
 y \nin \overline{(\varphi_2,\dots,\varphi_r)(Z_1)} \cup  \overline{(\varphi^{\prime}_2,\dots,\varphi^{\prime}_r)(Z^{\prime}_1)}
$
and show
$
 \bP(W_1) \times \{y\} \not \subset X.
$
By \pref{lm:q'(E'_1)}, the linear map
\begin{align}
 (s_2,\dots,s_{n+1})_y \colon V \to U \coloneqq \bigoplus_{i=2}^{n+1} \varpi_i^* T_{\bP(W_i)}(-1) \otimes k(y)
\end{align}
has rank $n$ and the restriction $(s_2,\dots,s_{n+1})_y |_{\ker s_1}$ has rank $n-1$.
Recall that the dimensions of $V$, $U$, and $\ker s_1$ are $n+1,n$, and $n-1$ respectively.
Hence $\ker (s_2,\dots,s_{n+1})_y \subset V$ is one-dimensional
and $\ker (s_2,\dots,s_{n+1})_y \cap \ker s_1 =\{0\} \subset V$.
Let $K \subset W_1$ be the image of $\ker (s_2,\dots,s_{n+1})_y$  by $s_1 \colon V \to W_1$.
Then we have a diagram
\begin{equation}
\begin{aligned}
  \xymatrix{  
    0 \ar[r] & \ker (s_2,\dots,s_{n+1})_y \otimes \cO \ar[r]  \ar[d]_{\wr}
    & V \otimes \cO  \ar[r]^(.5){ (s_2,\dots,s_{n+1})_y } \ar[d]^{\bss|_{\bP(W_1) \times \{y\} }} & U \otimes \cO  \ar[r] \ar@{=}[d]  & 0
    \\
     & K \otimes \cO   \ar[r]^(.4){(c,0)} &  T_{\bP(W_1)}(-1) \oplus (U \otimes \cO) \ar[r] & 
     U \otimes \cO \ar[r] & 0
  }
\end{aligned}
\end{equation}
on $\bP(W_1) \times \{y\}$,
where $c \colon K \otimes \cO \hookrightarrow W_1 \otimes \cO \to T_{\bP(W_1)}(-1) $ is the canonical map.
Then
$(c,0) $ is injective outside of the point $x_0 \in \bP(W_1)$
corresponding to the one-dimensional subspace $K \subset W_1$.
Hence the rank of $ \bss |_{\bP(W_1) \times \{y\} }$ is $n+1$ at any $x_1 \ne x_0 \in \bP(W_1)$,
which means that $(\bP(W_1) \setminus \{x_0\}) \times \{y\}  $ is not contained in $X$.
Thus $y$ is not contained in the left-hand side of \pref{eq:2_exc_locus}.
\end{proof}

\begin{proof}[Proof of \pref{th:main} (2)]
Lemmas \ref{lm:E_1_not_E_1'} and \ref{lm:2_exc_divs}
show that
$\bsqtilde(E_1) \subset X$ is one of the exceptional prime divisors of the birational morphism $X \to \prod_{i=2}^{n+1} \bP(W_i)$.
If we choose one of such divisors, we can reconstruct $\bsvarphi^{-1}$ or ${\bsvarphi^{\prime}}^{-1}$ by \pref{lm:inverse}
as in the proof of \pref{th:main} (1).
\end{proof}

\begin{remark}\label{rm:overR}
If $\varphi $ is defined over $\bR$,
so is $\varphi ^{\prime}$.
This follows from the construction of $\bss^{\prime}$ in (\ref{eq_def_s'}).
\end{remark}

We have the diagram
\begin{align}
\begin{aligned}
\xymatrix{
  & \Xtilde \ar[ld]_{ \ptilde} \ar[rd]^{\bsqtilde} \ar@{-->}[rr] & &  \Xtilde^{\prime}  \ar[ld]_{\bsqtilde^{\prime}} \ar[rd]^{ \ptilde^{\prime}} & \\ 
\bP(V) \ar@{-->}[rr]^{\bsvarphi} &   &  X =X^{\prime} & & \bP(V^{\prime}). \ar@{-->}[ll]_{\bsvarphi^{\prime}} \\
}
\end{aligned}
\end{align}
In the rest of this section, 
we describe the birational map ${\bsvarphi^{\prime}}^{-1} \circ \bsvarphi \colon \bP(V) \dto \bP(V^{\prime}) $.
Recall the definition of $L_i$ from \pref{eq:L_i}.

\begin{lemma}\label{lm:class}
The divisor $\bsqtilde(E_1) + \bsqtilde^{\prime}(E^{\prime}_1)$ on $X$
is linearly equivalent to $-L_1 + \sum_{i=2}^{n+1} L_i$.
\end{lemma}

\begin{proof}
Since the exceptional locus of the birational morphism
\begin{align}\label{eq:bir_morphism}
(\varpi_2,\dots,\varpi_{n+1}) |_X \colon X \to \prod_{i=2}^{n+1} \bP(W_i)
\end{align}
is $\bsqtilde(E_1) \cup \bsqtilde^{\prime}(E^{\prime}_1)$ by \pref{lm:2_exc_divs},
we can write
\begin{align} \label{eq:can_divs}
K_{X} = (\varpi_2,\dots,\varpi_{n+1})^* K_{\prod_{i=2}^{n+1} \bP(W_i)} + a \bsqtilde(E_1) +a^{\prime} \bsqtilde^{\prime}(E^{\prime}_1)
\end{align}
for some integers $a$ and $a^{\prime}$.
By the birational map
\begin{align}
 (\varphi_2,\dots,\varphi_{n+1}) \colon \bP(V) \dto \prod_{i=2}^{n+1} \bP(W_i),
\end{align}
the birational morphism
\pref{eq:bir_morphism} can be identified with the blow-up
$
 \Bl_{Z_1} \bP(V) \to \bP(V)
$
over the generic point of $Z_1$.
Hence the integer $a$ in \pref{eq:can_divs},
which is the coefficient of
$
 \bsqtilde(E_1) =\bP(W_1) \times \overline{(\varphi_2,\dots,\varphi_r)(Z_1)},
$
is one.
Similarly one has $a^{\prime}=1$, and
\begin{align}
 \bsqtilde(E_1) + \bsqtilde^{\prime}(E^{\prime}_1) = K_X -(\varpi_2,\dots,\varpi_{n+1})^* K_{\prod_{i=2}^{n+1} \bP(W_i)}
\end{align}
holds.
By \pref{eq:degeneracy},
the divisor $X=X_n \subset \prod_{i=1}^{n+1} \bP(W_i)$ is the zero locus of 
\begin{align}
 \lb \bsq_* \bss \rb^{\wedge (n+1)}
  \colon  \cO_{\prod_{i=1}^{n+1} \bP(W_i)} \simeq \bigwedge^{n+1} V \otimes \cO_{\prod_{i=1}^{n+1} \bP(W_i)} \to \bigwedge^{n+1} \bigoplus_{i=1}^{n+1} \varpi_i^* T_{\bP(W_i)}(-1) \simeq  \bigotimes_{i=1}^{n+1} \varpi_i^* \cO_{\bP(W_i)}(1).
\end{align}
Thus
$X$ is linearly equivalent to
$
 \sum_{i=1}^{n+1} \varpi_i^* \cO_{\bP(W_i)}(1)
$
on $\prod_{i=1}^{n+1} \bP(W_i)$.
Hence 
we have
$
 K_{X}= -\sum_{i=1}^{n+1} L_i
$
by the adjunction formula.
Since
$
 (\varpi_2,\dots,\varpi_{n+1})^* K_{\prod_{i=2}^{n+1} \bP(W_i)} = -2 \sum_{i=2}^{n+1} L_i,
$
one has
\begin{align}
 \bsqtilde(E_1) + \bsqtilde^{\prime}(E^{\prime}_1)
  &= K_X -(\varpi_2,\dots,\varpi_{n+1})^* K_{\prod_{i=2}^{n+1} \bP(W_i)}  = -L_1 + \sum_{i=2}^{n+1} L_i,
\end{align}
and \pref{lm:class} is proved.
\end{proof}

For each $i = 1, \ldots, n+1$,
let $F_i \subset \Xtilde$ be the strict transform of the divisor $\bsqtilde^{\prime}(E^{\prime}_i)  \subset X^{\prime} =X$.

\begin{lemma}\label{lm:classF}
The divisor $F_1$ 
is linearly equivalent to $\ptilde^* \cO_{\bP(V)}(n-1) - \sum_{i=2}^{n+1} E_i$.
\end{lemma}

\begin{proof}
This is an immediate consequence of
\pref{lm:class} and the linear equivalences
$
 \qtilde_i^* \cO_{\bP(W_i)}(1) \sim \ptilde^* \cO_{\bP(V)}(1) - E_i
$
for $1 \le i \le n+1$.
\end{proof}

\begin{corollary}
\label{cr:cremona}
\begin{enumerate}
\item The birational map
${\bsvarphi^{\prime}}^{-1} \circ \bsvarphi \colon \bP(V) \dto \bP(V^{\prime}) $
is obtained by the linear system
$
 \left| \cO_{\bP(V)}(n)  \otimes I_{\bigcup_{i=1}^{n+1} Z_i} \right|.
$
\item For each $i $,
the image $\ptilde \lb F_i \rb \subset \bP(V)$ is the unique hypersurface of degree $n-1$ containing
$Z_j$ for all $j \in \{ 1, \ldots, n+1 \} \setminus \{ i \}$.
\item
The birational map ${\bsvarphi^{\prime}}^{-1} \circ \bsvarphi $ contracts the hypersurface $\ptilde \lb F_i \rb$ to $Z^{\prime}_i$ for each $i$.
\end{enumerate}
\end{corollary}

\begin{proof}
 (1) By \pref{lm:inverse}, the birational map ${\bsvarphi^{\prime}}^{-1} $ is obtained by the linear system $\left| \cO_X(\bsqtilde^{\prime}(E^{\prime}_1)) \otimes  L_1 \right|$,
 which is identified with $\left| \cO_{\Xtilde}(F) \otimes \qtilde_1^* \cO_{\bP(W_1)}(1)  \right|$ by $\bsqtilde$.
Since  
 \begin{align}
 F_1  + \qtilde_1^* \cO_{\bP(W_1)}(1)
  &\sim \ptilde^* \cO_{\bP(V)}(n-1) - \sum_{i=2}^{n+1} E_i  +  \ptilde^* \cO_{\bP(V)}(1)  -E_1 \\
  &= \ptilde^* \cO_{\bP(V)}(n)  - \sum_{i=1}^{n+1} E_i
   \label{eq:O_V'(1)}
\end{align}
follows from \pref{lm:classF},
 $\left| \cO_{\Xtilde}(F) \otimes \qtilde_1^* \cO_{\bP(W_1)}(1)  \right|$
 is identified with 
 $
 \left| \cO_{\bP(V)}(n)  \otimes I_{\bigcup_{i=1}^{n+1} Z_i} \right|
$
by $\ptilde$.
Hence ${\bsvarphi^{\prime}}^{-1} \circ \bsvarphi = {\bsvarphi^{\prime}}^{-1} \circ \bsqtilde \circ \ptilde^{-1}$ is obtained by  
$
 \left| \cO_{\bP(V)}(n)  \otimes I_{\bigcup_{i=1}^{n+1} Z_i} \right|
$.
\\
(2)  It suffices to show this statement for $i=1$.
The linear system on $\bP(V)$
consisting of divisors of degree $n-1$
containing $Z_2, \ldots, Z_{n+1}$
is identified
with 
$\left| \ptilde^* \cO_{\bP(V)}(n-1) - \sum_{i=2}^{n+1} E_i \right|$
on $\Xtilde$ by $\ptilde$.
By \pref{lm:classF}, we have $\left| \ptilde^* \cO_{\bP(V)}(n-1) - \sum_{i=2}^{n+1} E_i \right| =|F_1|$,
which in turn is identified with the linear system $\left| E_1^{\prime} \right|$ on $\Xtilde^{\prime}$.
The linear system $\left| E_1^{\prime} \right|$
is 0-dimensional
since $E_1^{\prime}$ is an exceptional divisor.
Hence $\ptilde \lb F_1 \rb $ is the unique such divisor.
\\
(3) This statement holds since each $E^{\prime}_i \subset \Xtilde^{\prime}$ is contracted to $Z^{\prime}_i$.
\end{proof}

\section{Dominance for {$|\pmb{m}| \geq 2n-1$}}
 \label{sc:dominance}

We use the {\black same notation as} in \pref{sc:pr1}.

\begin{lemma}\label{lm:birational}
 The rational map
\begin{align}
\Phitilde {\black \colon}
 \prod_{i=1}^r \bP \lb V^{\vee} \otimes W_i \rb
  \dashrightarrow
 \Hilb \lb \bP^n \times \bP^{\bsm} \rb
{\black , \qquad}  \bsvarphi =  (\varphi_1, \dots, \varphi_r) \mapsto [ \Xtilde ]
\end{align}
 is birational onto an irreducible component.
\end{lemma}

\begin{proof}
Since we can recover $\bsvarphi$ from the graph $\Xtilde$ of $\bsvarphi$,
{\black the rational map} $\Phitilde$ is generically injective.
Hence it suffices to show that for general $\bsvarphi$,
$  \Hilb \lb \bP^n \times \bP^{\bsm} \rb$ is smooth at $\Phitilde(\bsvarphi ) $ and
the dimension of $  \Hilb \lb \bP^n \times \bP^{\bsm} \rb$ at $\Phitilde(\bsvarphi ) $
is equal to that of  $\prod_{i=1}^r \bP \lb V^{\vee} \otimes W_i \rb$.


Consider the differential
\begin{align} \label{eq:dPhi}
 \lb d \Phitilde \rb_{\bsvarphi} \colon
  T_{\bsvarphi} \lb \prod_{i=1}^r \bP \lb V^{\vee} \otimes W_i \rb \rb
   \to T_{[\Xtilde]} \Hilb \lb \bP^n \times \bP^{\bsm} \rb.
\end{align}
Let $\mathrm{Im} (\Phitilde) \subset \Hilb \lb \bP^n \times \bP^{\bsm} \rb $ be the closure of the image of $\Phitilde $
with reduced structure.
Then $ \lb d \Phitilde \rb_{\bsvarphi}$ factors as 
\begin{align}
 T_{\bsvarphi} \lb \prod_{i=1}^r \bP \lb V^{\vee} \otimes W_i \rb \rb
   \to T_{[\Xtilde]} \mathrm{Im} (\Phitilde) \subset T_{[\Xtilde]} \Hilb \lb \bP^n \times \bP^{\bsm} \rb.
\end{align}
We note that if $ \bsvarphi$ is general,
 $\mathrm{Im} (\Phitilde)$ is smooth at $[\Xtilde]$ and 
 $ T_{\bsvarphi} \lb \prod_{i=1}^r \bP \lb V^{\vee} \otimes W_i \rb \rb
   \to T_{[\Xtilde]} \mathrm{Im} (\Phitilde)$ is surjective.
Thus, for general $ \bsvarphi$, we have
\begin{align}
\rank \lb d \Phitilde \rb_{\bsvarphi} = \dim \mathrm{Im}  \lb\Phitilde \rb \leq \dim_{[\Xtilde]} \Hilb \lb \bP^n \times \bP^{\bsm} \rb \leq \dim T_{[\Xtilde]} \Hilb \lb \bP^n \times \bP^{\bsm} \rb.
\end{align}
If $ \lb d \Phitilde \rb_{\bsvarphi}$ is surjective,
$\rank \lb d \Phitilde \rb_{\bsvarphi} =\dim T_{[\Xtilde]} \Hilb \lb \bP^n \times \bP^{\bsm} \rb$,
and hence 
\begin{align}\label{eq:4terms}
\rank \lb d \Phitilde \rb_{\bsvarphi} = \dim \mathrm{Im}  \lb\Phitilde \rb = \dim_{[\Xtilde]} \Hilb \lb \bP^n \times \bP^{\bsm} \rb = \dim T_{[\Xtilde]} \Hilb \lb \bP^n \times \bP^{\bsm} \rb
\end{align}
holds. 
From the last equality of \pref{eq:4terms},
$  \Hilb \lb \bP^n \times \bP^{\bsm} \rb$ is smooth at $[\Xtilde]= \Phitilde(\bsvarphi ) $.
Furthermore, $ \dim_{[\Xtilde]} \Hilb \lb \bP^n \times \bP^{\bsm} \rb =\dim \mathrm{Im}  \lb\Phitilde \rb  =\dim \prod_{i=1}^r \bP \lb V^{\vee} \otimes W_i \rb$
since $\Phitilde$ is generically injective.
Thus this lemma follows from the surjectivity of $ \lb d \Phitilde \rb_{\bsvarphi}$ for general $\bsvarphi $.

In the rest of the proof,  we show the surjectivity of  $ \lb d \Phitilde \rb_{\bsvarphi}$ for general $\bsvarphi$.
Recall that $\Xtilde \subset \bP^n \times \bP^\bsm$ is the zero locus of a general section
$\bss \in  H^0 \lb \bP^n \times \bP^\bsm, \cE \rb$,
where $ \cE = \bigoplus_{i=1}^r p^* \cO_{\bP^n}(1) \otimes q_i^* T_{\bP^{m_i}}(-1) $.
Hence $\Xtilde$ is smooth with $\codim (X,\bP^n \times \bP^{\bsm}) = \rank \cE$,
 and the normal bundle $N_{\Xtilde/ \bP^n \times \bP^\bsm}$ is isomorphic to 
$\cE |_{\Xtilde} = \bigoplus_{i=1}^r \ptilde^* \cO_{\bP^n}(1) \otimes \qtilde_i^* T_{\bP^{m_i}}(-1)$.
We note that the isomorphism $N^{\dual}_{\Xtilde/ \bP^n \times \bP^\bsm} = I_{\Xtilde/ \bP^n \times \bP^\bsm} / I^2_{\Xtilde/ \bP^n \times \bP^\bsm} \cong \cE^{\dual} |_{\Xtilde} $
is induced from the surjective homomorphism $ \cE^{\dual} \twoheadrightarrow  I_{\Xtilde/ \bP^n \times \bP^\bsm} \subset \cO_{ \bP^n \times \bP^\bsm}$ defined by $\bss$.

From the Euler sequences on $\bP^{m_i}$'s,
we have an exact sequence on $\Xtilde$
\begin{align}
0 \rightarrow  \bigoplus_{i=1}^r \ptilde^* \cO_{\bP^n}(1) \otimes \qtilde_i^* \scO_{\bP^{m_i}}(-1) \rightarrow  \bigoplus_{i=1}^r \ptilde^* \cO_{\bP^n}(1) \otimes W_i  \rightarrow  \bigoplus_{i=1}^r \ptilde^* \cO_{\bP^n}(1) \otimes \qtilde_i^* T_{\bP^{m_i}}(-1) \rightarrow 0.
\end{align}
Since $\ptilde^* \cO_{\bP^n}(1) \otimes \qtilde_i^* \scO_{\bP^{m_i}}(-1) \simeq \scO_{\Xtilde}(E_i)$ and
$h^1(\Xtilde, \scO_{\Xtilde}(E_i)) =0 $, 
we see that the linear map 
\begin{align}\label{eq:tangent}
 \bigoplus_{i=1}^r V^{\vee} \otimes W_i =H^0 \left( \Xtilde, \bigoplus_{i=1}^r \ptilde^* \cO_{\bP^n}(1) \otimes W_i \right) \rightarrow H^0 \left(\Xtilde,  \bigoplus_{i=1}^r \ptilde^* \cO_{\bP^n}(1) \otimes \qtilde_i^* T_{\bP^{m_i}}(-1) \right)
\end{align}
is surjective with the kernel $ \bigoplus_{i=1}^r H^0( \Xtilde , \scO_{\Xtilde}(E_i) )  =\bigoplus_{i=1}^r  \bfk s_i $.
The induced isomorphism
\begin{align}
 \bigoplus_{i=1}^r V^{\vee} \otimes W_i / \bfk s_i
  \simto
 H^0 \lb \Xtilde,  \bigoplus_{i=1}^r \ptilde^* \cO_{\bP^n}(1) \otimes \qtilde_i^* T_{\bP^{m_i}}(-1) \rb
\end{align}
of vector spaces is identified
with the differential \pref{eq:dPhi}
under the isomorphisms
\begin{align}
 T_{\bsvarphi} \lb \prod_{i=1}^r \bP \lb V^{\vee} \otimes W_i \rb \rb
 \cong \lb \bigoplus_{i=1}^r (V^{\vee} \otimes W_i / \bfk s_i ) \otimes (\bfk s_i)^{\vee} \rb
  \simto \bigoplus_{i=1}^r V^{\vee} \otimes W_i / \bfk s_i
\end{align}
and
\begin{align}
 T_{[\Xtilde]} \Hilb \lb \bP^n \times \bP^{\bsm} \rb
  \cong H^0 \lb \Xtilde, N_{\Xtilde/ \bP^n \times \bP^\bsm} \rb
 \simto
 H^0 \lb \Xtilde,  \bigoplus_{i=1}^r \ptilde^* \cO_{\bP^n}(1) \otimes \qtilde_i^* T_{\bP^{m_i}}(-1) \rb 
\end{align}
induced by $s_i$'s.
This is a special case of the following general fact:
Let $Y$ be a projective variety,
$E$ be a vector bundle on $Y$,
$s$ be a global section of $E$, and
$Z=s^{-1}(0)$ be the zero locus of $s$.
When $Z$ is a complete intersection
(i.e., when $\codim Z=\rank E$),
the differential of
  \[
   \bP(H^0(E)) \dashrightarrow \Hilb(Y) : s \mapsto [Z]
  \]
can be identified with the natural morphism
  \[
  H^0(E)/\bfk s \rightarrow H^0(E|_Z).
  \] 
\end{proof}



\begin{proof}[Proof of \pref{th:main}.(\pref{it:dominance})]
Take general $\bsvarphi$.
We study the tangent space of $\Hilb (\bP^{\bsm} ) $ at $[X]$, which is isomorphic to $H^0(X, N_{X /  \bP^{\bsm}})$.
By $ | \bsm | \geq 2n-1$ and \pref{eq:dimX_n-1-2},
$\bsqtilde \colon \Xtilde \rightarrow X$ is an isomorphism in this case.
The diagram
\begin{equation}
\begin{aligned}
  \xymatrix{  
    0 \ar[r] & T_{ \Xtilde} \ar[r]  \ar[d]_{\wr}
    & T_{\bP^n \times \bP^{\bsm}} |_{\Xtilde} \ar[r] \ar[d] & N_{\Xtilde / \bP^n \times \bP^{\bsm}}  \ar[r] \ar[d]  & 0
    \\
    0 \ar[r] &  \bsqtilde^* T_X   \ar[r] &  \bsqtilde^* (T_{\bP^{\bsm}} |_{X}) \ar[r] & 
     \bsqtilde^* N_{X /  \bP^{\bsm}} \ar[r] & 0 
  }
\end{aligned}
\end{equation}
{\black induces}
an exact sequence
\begin{align}
0 \rightarrow \ptilde^* T_{\bP^n} \rightarrow N_{\Xtilde / \bP^n \times \bP^{\bsm}}  \rightarrow    \bsqtilde^* N_{X /  \bP^{\bsm}} \rightarrow 0
\end{align}
on $\Xtilde$.
Since
\begin{align}
 H^0 \lb \Xtilde, \ptilde^* T_{\bP^n} \rb
  \cong H^0 \lb \bP^n, T_{\bP^n} \rb
  \cong V^{\vee } \otimes V / \bfk \id_{V}
\end{align}
for
$
 \id_V \in \Hom (V,V) \cong V^{\vee } \otimes V
$
and
$
 h^1 ( \Xtilde, \ptilde^* T_{\bP^n} )  =0,
$
we have an exact sequence
\begin{align}\label{eq:tangent_map}
0 \rightarrow V^{\vee } \otimes V / \bfk \id_{V} \rightarrow  \bigoplus_{i=1}^r V^{\vee} \otimes W_i / \bfk s_i \stackrel{d}{\rightarrow} H^0(X,N_{X /  \bP^{\bsm}}) \rightarrow 0,
\end{align}
where {\black the middle term
\begin{align}
 H^0 \lb \Xtilde, N_{\Xtilde / \bP^n \times \bP^{\bsm}} \rb
  \cong \bigoplus_{i=1}^r V^{\vee} \otimes W_i / \bfk s_i
\end{align}
can be identified with
$
 T_{\bsvarphi} \lb \prod_{i=1}^r \bP \lb V^{\vee} \otimes W_i \rb \rb
$
as in the proof of \pref{lm:birational}.
Then the map $d$ in \pref{eq:tangent_map} can be identified
with $(d \Phi)_{\bsvarphi}$,
and $V^{\vee } \otimes V / \bfk \id_{V} $ can be identified
with} the tangent space of the $\PGL(n+1, \bfk)$-orbit of $\bsvarphi$.
By \pref{th:main}.(\pref{it:general}),(\pref{it:tolines}),
a general fiber of $\Phi$ consists {\black of} at most two $\PGL(n+1, \bfk)$-orbits.
{\black Note} that the dimension of the $\PGL(n+1, \bfk)$-orbits is equal
to that of $\PGL(n+1, \bfk)$ since $\bsvarphi$ is birational.
Hence {\black we have}
$
 \dim  \im(\Phi) = \dim \prod_{i=1}^r \bP(V^{\vee} \otimes W_i) - \dim \PGL(n+1, \bfk).
$
Thus $\dim \im (\Phi)  $ is equal to $h^0(X,N_{X /  \bP^{\bsm}})$ by \pref{eq:tangent_map},
which is the dimension of the tangent space of $\Hilb (\bP^{\bsm} )$ at $[X]$.
This means that $\Hilb (\bP^{\bsm} )$ is smooth at $[X]$ of dimension $h^0(X,N_{X /  \bP^{\bsm}}) = \dim \im (\Phi) $.
Hence $\Phi$ is dominant onto an irreducible component.
\end{proof}

\begin{remark}
				In the case $n = 3$ and $\bsm = (2^r)$,
				the condition $ | \bsm | \geq 2n-1$ is $2r \geq 5,$ 
that is $r \geq 3$.
As in \cite[Section 6]{MR3095002},
the closure of the image of $\Phi$ is a cubic hypersurface in an irreducible component $\bP^8=\bP (W_1^{\vee} \otimes W_2^{\vee}) \subset \Hilb \lb \bP^{\bsm} \rb $ for $r=2$,
so
this is sharp in this case.
We do not know whether the condition $ | \bsm | \geq 2n-1$ is sharp or not in general.
\end{remark}

\section{Grassmann tensors as Chow forms}
\label{sc:hartley}

We recall the standard methods for projective reconstruction from a geometric point of view.
Define an index set 
\begin{align}
 B(n,\bsm) \coloneqq \lc \bsalpha = (\alpha_1,\dots, \alpha_r) \in \bN^r \relmid
 1 \le \alpha_i \le m_i \text{ for any } i=1,\dots, r 
 \text{ and } \sum_{i=1}^r \alpha_i = n+1 \rc.
\end{align}
For an $n$-dimensional variety $X \subset \bP^{\bsm}$ 
and $\bsalpha \in B(n,\bsm)$,
let
\begin{align}
 \scZ(X) \subset G(\bsm-\bsalpha,\bP^{\bsm}) 
 \coloneqq \prod_{i=1}^r G(m_i-\alpha_i,\bP(W_i))
\end{align}
denote the set of points in $G(\bsm-\bsalpha,\bP^{\bsm})$ 
corresponding to $r$-tuples $(U_1,\dots U_r)$ of linear subvarieties 
$U_i \subset \bP(W_i)$ of codimension $\alpha_i$ such that
\begin{align}
 \label{eq:equiv1}
  X \cap \prod_{i=1}^r U_i \ne \emptyset.
\end{align}
Here, $G(m_i-\alpha_i,\bP(W_i))$ is the Grassmannian of
$(m_i-\alpha_i)$-planes in $\bP(W_i)$,
which is embedded in 
$\bP\lb \bigwedge^{m_i+1-\alpha_i}W_i \rb = \bP\lb \bigwedge^{\alpha_i}W_i^\vee \rb$
by the Pl\"ucker embedding for $i=1,\dots, r$.

\begin{theorem}[{\cite[Theorem 3.1]{Hartley2009}, \cite[Theorem 2]{6909462}}] 
 \label{th:equiv}
  Assume that $\bsvarphi$ is generic.
 For $\alpha \in B(n,\bsm)$ and the multiview variety $X \subset \bP^{\bsm}$,
 the set $\scZ(X) \subset G(\bsm-\bsalpha,\bP^{\bsm}) $ is a hypersurface defined by the multilinear equation
 \begin{align} \label{eq:equiv2}
  \sum_{\sigma_1,\dots, \sigma_r} A^{\sigma_1,\dots, \sigma_r}
  p^1_{\sigma_1}\dots p^r_{\sigma_r} = 0,
\end{align}
in $G(\bP^{\bsm},\bsalpha)$,
where $p^i=\ld p^i_{\sigma_i} \mid 
 \sigma_i=(\sigma_{i,1},\dots, \sigma_{i,\alpha_i}), 
 1 \le \sigma_{i,1} < \dots 
 <\sigma_{i,\alpha_i} \le m_i +1\rd$
is the Pl\"ucker coordinates 
 of the Grassmannian $G(m_i-\alpha_i , \bP(W_i))$ for each $i=1,\dots, r$.
Moreover, the tensor 
\begin{align}
 A \coloneqq \lb A^{\sigma_1,\dots, \sigma_r} \rb 
 \in \bigotimes_{i=1}^r \bigwedge^{\alpha_i} W_i,
\end{align}
is uniquely determined up to scalar
from sufficiently many 
 \emph{subspace correspondences}.
 \footnote{Just as a point on the multiview variety $X$ is called a point correspondence,
 a point in $\cZ(X)$ is called a \emph{subspace correspondence}.}
\end{theorem}
The tensor $A$, called the \emph{Grassmann tensor of profile $\alpha$} 
for a camera configuration $\bsvarphi$
in computer vision,
is an analog of the \emph{Chow form}
of the multiview variety $X \subset \bP^{\bsm}$ 
(see, e.g., \cite[Section 3.2]{MR2394437}).
Note that the Grassmann tensors of profile $(2,2), (2,1,1)$, and $(1,1,1,1)$ 
are classically known as the fundamental matrix, the trifocal tensor,
and the quadrifocal tensor, respectively.

The \emph{projective reconstruction
theorem} by Hartley and Schaffalitzky can be stated as follows:
\begin{theorem}[{\cite[Section 5]{Hartley2009}}] 
 \label{th:original}
 Assume that $\bsvarphi$ is generic.
 Fix $\alpha \in B(n,\bsm)$ and let $A$ be
 a Grassmann tensor of profile $\alpha$.
\begin{enumerate}
 \item 
If $\bsm \ne (1^{n+1})$,
then $\bsvarphi$ is uniquely determined by $A$ up to the actions of $\PGL(n+1, \bfk)$.
 \item 
If $\bsm = (1^{n+1})$,
then two candidates of $\bsvarphi$ are obtained by $A$ up to the actions of $\PGL(n+1, \bfk)$.
\end{enumerate}
\end{theorem}

Since the multiview variety $X$
determines its associated hypersurface $\cZ(X)$
and the Grassmannian tensor $A$,
\pref{th:original} implies \pref{th:main}.(1),(2).

An algorithm for the projective reconstruction
for general $n$ and $\bsm$
using Grassmann tensors
has been implemented
in \cite{Hartley2009},
and applied
to the analysis of 
dynamic scenes
\cite{Wolf2002, HV2008}.
In contrast,
our proof of the projective reconstruction theorem is based
on the analysis of divisors on the multiview variety,
and it is an interesting problem to see if it can lead to a new algorithm
for the reconstruction.

\bibliographystyle{amsalpha}
\bibliography{reconstruction}
\end{document}